\documentclass[reqno, oneside, 12pt]{amsart}

\usepackage[letterpaper]{geometry}
\usepackage{amsmath}
\usepackage{amssymb}
\usepackage{enumerate}
\usepackage{verbatim}
\usepackage{mathrsfs}
\usepackage{mathtools}

\mathtoolsset{centercolon}

\usepackage{setspace}

\numberwithin{equation}{section}

\newtheorem{thm}{Theorem}

\newtheorem{lem}[thm]{Lemma}
\newtheorem*{lem*}{Lemma}

\theoremstyle{remark}
\newtheorem*{rem*}{Remark}
\newtheorem*{rems*}{Remarks}
 
\theoremstyle{definition}

\renewcommand{\tilde}{\widetilde}
\renewcommand{\hat}{\widehat}

\newcommand{\Z}{\mathbb{Z}}
\newcommand{\Q}{\mathbb{Q}}
\newcommand{\N}{\mathbb{N}}
\newcommand{\C}{\mathbb{C}}

\renewcommand{\H}{\mathbb{H}}

\renewcommand{\pmatrix}[4]{\left( \begin{smallmatrix} #1 & #2 \\ #3 & #4 \end{smallmatrix} \right)}
\newcommand{\pMatrix}[4]{\left( \begin{matrix} #1 & #2 \\ #3 & #4 \end{matrix} \right)}

\newcommand{\pfrac}[2]{\left(\frac{#1}{#2}\right)}
  
\DeclareMathOperator{\ord}{ord}
\DeclareMathOperator{\lcm}{lcm}

\date{}

\author{Nickolas Andersen}
\address{Department of Mathematics, Univeristy of Illinois at Urbana-Champaign}
\email{nandrsn4@illinois.edu}

\author{Holley Friedlander}
\address{Department of Mathematics, University of Massachusetts, Amherst}
\email{holleyf@math.umass.edu}

\author{Jeremy Fuller}
\address{Department of Mathematics, Purdue University}
\email{jtfuller@math.purdue.edu}

\author{Heidi Goodson}
\address{Department of Mathematics, University of Minnesota}
\email{goods052@umn.edu}

\title[Mock Theta Congruences]{Effective Congruences for Mock Theta Functions}

\begin{document}

\begin{abstract}
  Let  $M(q)=\sum c(n)q^n$ be one of Ramanujan's mock theta functions. We establish the existence of infinitely many linear congruences of the form
  \[
  	c(An+B) \equiv 0 \pmod{\ell^{j}},
  \]
where $A$ is a multiple of $\ell$ and an auxiliary prime $p$. Moreover, we give an effectively computable upper bound on the smallest such $p$ for which these congruences hold. The effective nature of our results is based on prior works of Lichtenstein \cite{Lichtenstein} and Treneer \cite{Treneer}.
\end{abstract}

\maketitle

\section{Introduction and Statement of the Results}
\label{sec:introduction_and_statement_of_the_results}

A {\it partition} of a positive integer $n$ is a non-increasing sequence of positive integers that sum to $n$.  Define $p(n)$ to be the number of partitions of a non-negative integer $n$.  Ramanujan \cite{Ramanujan} proved the linear congruences
\begin{eqnarray*}
	p(5n+4)&\equiv& 0 \pmod{5},\\
	p(7n+5)&\equiv& 0 \pmod{7},\\
	p(11n+6)&\equiv& 0 \pmod{11},
\end{eqnarray*}
which were later extended by Atkin \cite{Atkin1} and Watson \cite{Watson} to include powers of $5,7,11$. Later, Atkin \cite{Atkin2} developed a method to identify congruences modulo larger primes, such as
\begin{eqnarray*}
	p(17303n+237)&\equiv&0 \pmod{13},\\
	p(1977147619n+815655)&\equiv&0 \pmod{19},\\
	p(4063467631n+30064597)&\equiv& 0 \pmod{31}.
\end{eqnarray*}
Ahlgren and Ono \cite{Ahlgren, Ahlgren:Ono, Ono:2000} have shown that linear congruences for $p(n)$ exist for all moduli $m$ coprime to $6$.

These congruences arise from studying the arithmetic properties of the generating function
\[
	\displaystyle\sum_{n=0}^{\infty} p(n)q^n = \prod_{n=1}^{\infty} \frac{1}{1-q^n} = 1 +q+ 2q^2+3q^3+5q^4+7q^5+ \ldots,
\]  
which can also be writen in Eulerian form
\[
	\displaystyle\sum_{n=0}^{\infty} p(n)q^n =1+\sum_{n=1}^{\infty} \frac{q^{n^2}}{(q;q)_{n}^{2}},
\]
where $(x;q)_n:=(1-x)(1-xq)(1-xq^2)\cdots(1-xq^{n-1})$. By changing signs, we obtain one of Ramanujan's third-order mock theta functions
\[
	f(q)=\displaystyle\sum_{n=0}^{\infty} a_f(n)q^n:=1+\sum_{n=1}^{\infty} \frac{q^{n^2}}{(-q;q)_{n}^{2}} = 1+q-2 q^2+3 q^3-3 q^4+3 q^5 + \ldots.
\]
The coefficients $a_f(n)$ of $f(q)$ can be used to determine the number of partitions of $n$ of even rank and of odd rank \cite{Bringmann:Ono}.

The function $f(q)$ is one of Ramanujan's seventeen original mock theta functions, which are \emph{strange} $q$-series that often have combinatorial interpretations (see \cite{McIntosh} for a comprehensive survey of mock theta functions). These functions have been the source of much recent study. In \cite{Alfes, Bruinier:Ono, BPOR, Chan, Garthwaite, Waldherr}, congruences for the coefficients of various mock theta functions are established. For example, in their investigation of strongly unimodal sequences, Bryson, Ono, Pitman, and Rhoades \cite{BPOR} prove the existence of congruences for the coefficients of Ramanujan's mock theta function
\[
	\Psi(q) = \displaystyle\sum_{n=1}^{\infty}a_{\Psi}(n)q^n := \displaystyle\sum_{n=1}^{\infty}\frac{q^{n^2}}{(q;q^2)_n} = q + q^2 + q^3 + 2 q^4 + 2 q^5 + \ldots.
\]
In particular, they establish the congruence
\begin{equation} \label{eqn:psi_congr}
	a_{\Psi}(11^4 \cdot 5n+721)\equiv 0 \pmod{5}. 
\end{equation}
In \cite{Waldherr}, Waldherr shows that Ramanujan's mock theta function
\[
	\omega(q) = \sum_{n=0}^{\infty} a_{\omega}(n)q^n := \displaystyle\sum_{n=0}^{\infty} \frac{q^{2n^2+2n}}{(q;q^{2})_{n}^{2}} = 1 + 2 q^4 + 2 q^5 + 2 q^6 + \ldots
\]
satisfies
\[
	a_{\omega}(40n+27)\equiv a_{\omega}(40n+35) \equiv 0 \pmod{5}.
\]

Congruences like the examples above have also been proven for other mock theta functions, such as Ramanujan's $\phi(q)$ function \cite{Bruinier:Ono, Chan}. It is natural to ask if a general theory of such congruences exists.  In this paper, we build on the approaches of these previous works to establish the existence of linear congruences for all of Ramanujan's mock theta functions. 

If $M(q)$ is one of Ramanujan's mock theta functions, then by work of Zwegers \cite{Zwegers} there are relative prime integers $\delta$ and $\tau$ for which 
\begin{equation} \label{eq:F-M}
	F(z) = \sum  b(n) q^{n} := q^\tau M(q^{\delta})
\end{equation}
is the holomorphic part of a weight 1/2 harmonic weak Maass form (to be defined in Section \ref{nuts_and_bolts}). We obtain congruences for the coefficients of $M(q)$ as in \eqref{eqn:psi_congr} by obtaining them for $F(z)$.

\begin{thm}\label{thm:mocktheta}
	Let $M(q)$ be one of Ramanujan's mock theta functions with $F(z)$ as in \eqref{eq:F-M}, let $N$ be the level of $F$, and let $\ell^j$ be a prime power with $(\ell, N)=1$.  Then there is a prime $Q$ and infinitely many primes $p$ such that, for some $m,B\in \N$, we have 
\[
	b(p^4 \ell^m Q n + B) \equiv 0 \pmod{\ell^j}.
\]
Furthermore, the smallest such $p$ satisfies $p\leq C$, where $C$ is an effectively computable constant that depends on $\ell^j$, $N$, and other computable parameters.
\end{thm}

\begin{rem*}
Theorem \ref{thm:mocktheta} is a special case of Theorem \ref{thm:main}, a more general result that applies to weight 1/2 harmonic Maass forms whose holomorphic parts have algebraic coefficients and whose nonholomorphic parts are period integrals of weight $3/2$ unary theta series.
\end{rem*}

\section*{Acknowledgements}

The authors would like to thank the Southwest Center for Arithmetic Geometry for supporting this research during the 2013 Arizona Winter School.  We would also like to thank Ken Ono for suggesting this project and for guiding us through the process of writing this paper.

\section{Nuts and Bolts} \label{nuts_and_bolts}
The proof of Theorem \ref{thm:main} utilizes several important concepts from the theory of modular forms and harmonic Maass forms. In this section we summarize those topics and results that will be key ingredients in the proof. 
\subsection{Harmonic Maass Forms} 
Ramanujan's mock theta functions are essentially the holomorphic parts of certain weight 1/2 harmonic Maass forms. To begin, we define half integral weight harmonic weak Maass forms.
Here ``harmonic'' refers to the fact that these functions vanish under the weight $k$ hyperbolic Laplacian $\Delta_k$, 
\begin{equation}\Delta_k:=-y^2\left(\frac{\partial^2}{\partial x^2}+\frac{\partial^2}{\partial y^2}\right)+iky\left(\frac{\partial}{\partial x}+\frac{\partial}{\partial y}\right),\end{equation}
for $z=x+iy \in \H$.

If $N$ is a positive integer with $4|N$ and $\chi$ a Dirichlet character modulo $N$, a weight $k\in\frac{1}{2}\Z$ {\em harmonic weak Maass form} for a congruence subgroup $\Gamma \in \{\Gamma_1(N),\Gamma_0(N)\}$, with Nebentypus $\chi$, is any smooth function $f:\H \to \C$ satisfying

\begin{enumerate}
\item For every $\gamma=\pmatrix{a}{b}{c}{d} \in \Gamma,$ we have
\[
f\left(\frac{az+b}{cz+d}\right)=\pfrac{c}{d}^{2k} \varepsilon_d^{-2k} \, \chi(d) \, (cz+d)^k \, f(z),
\]
where 
\[
	\varepsilon_d:= 
	\begin{cases}
		 1&\mbox{if } d \equiv 1 \mod 4, \\ 
		 i &\mbox{if } d \equiv 3 \mod 4.
	\end{cases}
\]

\item We have $\Delta_kf=0$. 
\item There is a polynomial $P_f=\sum_{n \geq 0} c^+(n)q^n \in \C[q^{-1}]$ such that 
\(
f(z)-P_f(z)=O(e^{-\epsilon y})
\)
as $y \to +\infty$ for some $\epsilon>0$. Analogous conditions are required at all cusps.
\end{enumerate}
The term ``weak'' refers to the relaxed growth condition at the cusps described by (3). For convenience, we will refer to these harmonic weak Maass forms simply as harmonic Maass forms.

We adopt the following notation: if $\chi$ is a Dirichlet character modulo $N$, let $S_k(N,\chi)$ (respectively, $ M_k(N,\chi),  M_k^!(N,\chi), H_k(N,\chi)$) denote the space of cusp forms (resp., holomorphic modular forms, weakly holomorphic modular forms, harmonic Maass forms) of weight $k$ on $\Gamma_{0}(N)$ with Nebentypus $\chi$.

For $f \in H_{2-k}(N,\chi)$ and $1<k\in \frac{1}{2}\Z$, we write $f=f^++f^-$, and define 
\[
f^{+}(z):=\sum_{n \gg -\infty} a^+(n)q^n
\] 
to be the {\em holomorphic part} and
\[
f^{-}(z):=\sum_{n<0}a^-(n)\Gamma(k-1,4\pi|n|y)q^n
\]
to be the {\em nonholomorphic part}, where $\Gamma(a,b):=\int_b^\infty e^{-t}t^{a-1}\,dt$ is the incomplete Gamma-function.

If $M(q)$ is one of Ramanujan's mock theta functions with $F(z)$ as in \eqref{eq:F-M}, then by \cite{Zwegers}, $f^{+}=F$ is the holomorphic part of a weight 1/2 harmonic weak Maass form whose nonholomorphic part $f^{-}$ is a period integral of a weight 3/2 unary theta series. As a consequence, there exist integers $\delta_1,\dots,\delta_h$ such that the coefficients $a^-(n)$ are supported on exponents of the form $-\delta_i m^2$.

As stated in Section \ref{sec:introduction_and_statement_of_the_results}, Theorem \ref{thm:mocktheta} is a special case of Theorem \ref{thm:main}, which applies to weight $1/2$ harmonic Maass forms with algebraic coefficients whose nonholomorphic parts are period integrals of weight 3/2 unary theta series. Essentially, these congruences are obtained from the annihilation of a cusp form $g(z)$, related to $f(z)$, by the Hecke operators $T(p^2)$. The cusp form $g(z)$ is determined by a result of Treener \cite{Treneer}. Moreover, work of Lichtenstein \cite{Lichtenstein} allows us to bound the first prime $p$ such that $T(p^2)$ annihilates $g(z)$. The details of the construction of $g(z)$ follow.

\subsection{Elements of the Proof}
To prove Theorem \ref{thm:main}, we first obtain a weakly holomorphic modular form $\hat{f}(z)$ by applying quadratic twists to annihilate the nonholomorphic part of $f(z)$. If $Q$ is an odd prime, define $\psi_{Q} := \pfrac{\bullet}{Q}$ and
\(
	G := \sum_{\lambda = 1}^{Q-1} \psi_{Q}(\lambda) e^{2\pi i \lambda/Q}.
\)
Then the {\em $Q$-quadratic twist of $f$} is defined as
\[
	f\otimes \psi_{Q} := \frac{G}{Q} \sum_{\lambda = 1}^{Q-1} \psi_{Q}(\lambda) f \big|_{\frac{1}{2}} \pMatrix{1}{-\frac{\lambda}{Q}}{0}{1}.
\]

\begin{rems*}
The definition of $f\otimes \psi_{Q}$ given in \cite[III, Proposition 17]{Koblitz} applies to modular forms, but this definition also makes sense for $f\in H_{2-k}(N,\chi)$ since the transformation $z \mapsto z-\lambda/Q$ only affects the real part of $z$ (the $\Gamma$-factor in $f^{-}$ remains unchanged). As in the modular case (see \cite[III, Proposition 17]{Koblitz}), the $n$th coefficient of $f\otimes \psi_{Q}$ is $\pfrac{n}{Q}$ times the $n$th coefficient of $f$.
\end{rems*}

The following lemma describes how twisting $f$ affects the level:

\begin{lem}\label{twistlevel}
Suppose $f$ satisfies the transformation
\(
	f |_{k} \pmatrix{a}{b}{c}{d} = \chi(d) f
\)
for all $\pmatrix{a}{b}{c}{d} \in \Gamma_{0}(N)$ and for some character $\chi$ mod $N$. Let $\psi$ be a character mod $M$, and let $N' = \lcm(NM, M^{2})$. Then
\[
	(f \otimes \psi) \big|_{k} \pMatrix{a}{b}{c}{d} = \chi(d) \psi^{2}(d) (f \otimes \psi)
\]
for all $\pmatrix{a}{b}{c}{d} \in \Gamma_{0}(N')$.
\end{lem}

\begin{proof}
	Let $\pmatrix{a}{b}{c}{d} \in \Gamma_{0}(N')$. For each $\lambda$ with $0\leq \lambda < M$, let $\lambda'$ denote the smallest nonnegative integer satisfying $\lambda' a \equiv \lambda d \pmod{M}$. Then we have
	\[
		\pMatrix{1}{-\lambda/M}{0}{1} \pMatrix{a}{b}{c}{d} \pMatrix{1}{-\lambda/M}{0}{1}^{-1} = \pMatrix  {a-\frac{c \lambda }{M}} {b+\frac{a \lambda' - d \lambda }{M}-\frac{c \lambda \lambda '}{M^2}} {c} {d+\frac{c \lambda '}{M}} \in \Gamma_{0}(N').
	\]
	The lemma now follows from a standard argument (see \cite[Proposition III.17(b)]{Koblitz}).
\end{proof}

In the proof of Theorem \ref{thm:main}, we require a cusp form $g(z)$ with the property that $\hat{f}(z) \equiv g(z) \pmod{\ell^j}$. The existence of such a cusp form is guaranteed by work of Treneer \cite{Treneer}. We first fix some notation. For $f\in M_{k}^{!}(N,\chi)$ and a prime $\ell$, define $\alpha=\alpha(f,\ell)$ and $\beta=\beta(f,\ell)$ to be the smallest nonnegative integers satisfying
\begin{equation} \label{def:alpha-beta}
		-\ell^{\alpha} < 4\min_{\ell^{2}|c} \{ \ord_{\frac{a}{c}} \hat{f} \}, \qquad
		-\ell^{\beta} <  \min_{\ell^{2} \nmid c} \{ \ord_{\frac{a}{c}} \hat{f} \},
\end{equation}
where $\frac{a}{c}$ runs over a set of representatives for the cusps of $\Gamma_{0}(N)$. Theorem \ref{thm:treneer} below follows from Theorems 1.1 and 3.1 of \cite{Treneer}, along with the proof of Theorem 3.1 of \cite{Treneer}.

\begin{thm}{\cite[Theorem 3.1]{Treneer}}\label{thm:treneer}
Let $\ell^{j}$ be an odd prime power and let $N$ be a positive integer with $(\ell,N)=1$. Suppose that $f(z)=\sum a(n)q^n \in M_{\frac{k}{2}}^{!}(4N,\chi)$ has algebraic integer coefficients. If $\alpha=\alpha(f,\ell)$ and $\beta=\beta(f,\ell)$ as in \eqref{def:alpha-beta}, then there is a cusp form 
\[
	g(z) \in S_{\frac{k}{2}+\frac{\ell^\beta(\ell^2-1)}{2}}(\Gamma_0(N\ell^2),\chi\psi_\ell^{k\alpha})
\]
such that
\[
	g(z) \equiv \sum_{\ell \nmid n} a(\ell^\alpha n)q^n \pmod {\ell^j}.
\]
Further, a positive proportion of primes $p \equiv -1 \pmod{4N \ell^{j}}$ satisfy
\[
	a(p^{3} \ell^{\alpha} n)\equiv 0 \pmod{\ell^{j}}
\]
for all $n$ coprime to $\ell p$.
\end{thm}

We end this section by recalling the action of Hecke operators $T(p^2)$ on half integral weight cusp forms and stating a result of Lichtenstein \cite{Lichtenstein}, which will allows us to bound the smallest prime $p$ such that $g(z)\mid T(p^2) \equiv 0 \pmod{\ell^j}$. If $\chi$ is a quadratic character, $g(z) = \sum a(n) q^{n} \in M_{\lambda+1/2} (4N,\chi),$ and $(p,4N)=1$ then
\[
	g(z) \mid T(p^{2}) := \sum \left( a(p^{2}n) + \pfrac{(-1)^{\lambda}n}{p} \chi(p) \, p^{\lambda-1} a(n) + p^{2\lambda-1} a(n/p^{2}) \right) q^{n}.
\]
To state Lichtenstein's result, we require the following notation. Let $E$ be the smallest number field containing the coefficients of $g(z)$ and let $\ell \mathcal{O}_E$ factor as $\ell \mathcal{O}_E=\prod_m \lambda_m^{e_{m}}$, where the $\lambda_m$ are prime ideals of $\mathcal{O}_E$. Let $S:=S_k(\Gamma_0(2N))$ and set $d:=\dim_{\mathbb{C}} S$. Let 
\[
	s:=\frac{k}{12} N\prod_{p| N} \left(1+\frac{1}{p}\right)
\]
denote the Sturm bound for $S$. Let $\{f_1,\dots,f_d\}$ be a basis for $S$ consisting of normalized Hecke eigenforms, and let $K$ be a number field containing the coefficients of all the $f_{i}$. Choose primes $\mu_m$ of $\mathcal{O}_{KE}$ lying above each $\lambda_m$ such that the largest $\mu_m$-adic valuation of the first $s$ coefficients of all the $f_i$ is a minimum -- let $v_m$ denote this largest valuation. Let $r_m:=[\mathcal{O}_{KE}/\mu_m:\mathbb{F}_\ell]$. Define 
\[
	B(S,\ell^j):=\prod_m \ell^{4dr_m(dv_m+\alpha_m)}
\]
and $L(S,\ell):=\ell\cdot \prod_{p| N} p$.

\begin{thm}{\cite[Theorem 1.2]{Lichtenstein}}\label{thm:lichtenstein}
  With the notation above, let $B:=B(S,\ell^j)$ and $L:=L(S,\ell)$. There is an effectively computable constant $A_1$ (defined in \cite{LMO}) such that for some prime $p\equiv -1\pmod{2N\ell^j}$ satisfying 
\[
	p\leq 2(L^{(B-1)} B^{B})^{A_{1}}
\] 
we have $g\mid T(p^{2})\equiv 0\pmod {\ell^j}$. Assuming the Generalized Riemann Hypothesis, the prime $p$ satisfies 
\[
	p\leq 280 B^2(\log B+\log L)^2.
\]
\end{thm}

\begin{rem*}
  In particular, if the coefficients of $f$ are rational integers, i.e. $E=\mathbb{Q}$, then taking $\mu=\mu_1$, $v=v_1$, and $r=r_1$, we see that $B=\ell^{4dr(dv+j)}$.
\end{rem*}

\begin{rem*}
	The quantity $B$ defined in Theorem \ref{thm:lichtenstein} arises from the elementary bound
\[
	\# \textup{GL}_{2}(\mathcal{O}_{K}/\mu^{dv+j}) \leq \ell^{4r(dv+j)}
\]
which, in certain cases, is easy to compute and is much smaller (see \cite[Example 4.3]{Lichtenstein}).
\end{rem*}

\section{Statement of the General Theorem and its Proof}
\label{sec:statement_of_the_general_theorem_and_its_proof}

Here we state our general result of which Theorem \ref{thm:mocktheta} is a special case. For ease of notation, we state it for harmonic Maass forms with holomorphic parts whose coefficients lie in $\Q$, but an analogous result holds for such forms with algebraic coefficients.

\begin{thm}\label{thm:main} 
	Suppose $f \in H_{\frac{1}{2}} (\Gamma_{0}(4N),\chi)$ has holomorphic part
\(
	f^{+} = \sum_{n\geq n_{0}} a^{+}(n) q^{n}
\) with $a^{+}(n) \in \Q$ 
and non-holomorphic part 
\[
	f^{-} = \sum_{i=1}^{h} \sum_{n > 0} a^{-}(\delta_{i} n^{2}) \Gamma\left(\frac{1}{2},4\pi\delta_{i}n^{2}\right)q^{-\delta_{i}n^{2}}
\]
for some finite set of squarefree $\delta_{i} \in \N$. 
Let $\ell^{j}$ be a prime power with $(4N,\ell)=1$. 
\begin{enumerate}[\textup{(}i\textup{)}]
\item Let $Q$ be an odd prime with $\pfrac{\delta_{i}}{Q}=1$ for $1\leq i\leq h$. Then we have
\[
	\hat{f} = \sum \hat{a} (n) q^{n} := \sum_{\pfrac{-n}{Q} = -1} a^{+}(n) q^{n} \in M_{\frac{1}{2}}^{!}(\Gamma_{0}(4NQ^{3}),\chi).
\]
\item Define $\alpha=\alpha(\hat{f},\ell)$ and $\beta=\beta(\hat{f},\ell)$ as in \eqref{def:alpha-beta}. Then there exists a cusp form
\[
	g \in S_{\frac{1}{2} + \ell^{\beta}\pfrac{\ell^{2}-1}{2}} (\Gamma_{0}(4NQ^{3}\ell^{2}),\chi\psi_{\ell}^{\alpha})
\]
with the property that
\[
	g \equiv \sum_{\ell\nmid n} \hat{a}(\ell^{\alpha}n) q^{n} \pmod{\ell^{j}}.
\]
Further, a positive proportion of the primes $p \equiv -1 \pmod{4NQ^{3}\ell^{j}}$ have
\begin{equation} \label{eqn:a-hat-cong}
	\hat{a}(p^{3}\ell^{\alpha}n) \equiv 0 \pmod{\ell^{j}}
\end{equation}
for all $n$ coprime to $\ell p$.
\item Define
\(
	S := S_{\ell^{\beta}(\ell^{2}-1)} (\Gamma_{0}(2NQ^{3}\ell^{2})),
\) and let $B=B(S,\ell^{j})$ and $L=L(S,\ell)$ as given in Theorem \ref{thm:lichtenstein} above. Then the smallest prime $p$ for which \eqref{eqn:a-hat-cong} holds satisfies
\[
	p \leq 2 (L^{B-1}B^{B})^{A_{1}}.
\]
Assuming GRH, this prime $p$ satisfies
\[
	p \leq 280 B^{2}(\log B + \log L)^{2}.
\]
\end{enumerate}
\end{thm}

\begin{proof}[Proof of Theorem \ref{thm:main}]
	To prove ($i$), we use $Q$-quadratic twists to annihilate the nonholomorphic part of $f$. We have
\begin{align*}
	f^{-} - \pfrac{-1}{Q}f^{-} \otimes \psi_{Q}
	&= \sum_{j=1}^{s} \sum_{n=1}^{\infty} \left( 1 - \pfrac{-1}{Q}\pfrac{-\delta_{j}n^{2}}{Q} \right) a^{-}(\delta_{j}n^{2}) \Gamma(1/2,4\pi \delta_{j} n^{2} y) q^{-\delta_{j}n^{2}} \\
	&=\sum_{j=1}^{s} \sum_{Q | n} a^{-}(\delta_{j}n^{2}) \Gamma(1/2,4\pi \delta_{j} n^{2} y) q^{-\delta_{j}n^{2}}.
\end{align*}
Let $\tilde{f}:= f-\pfrac{-1}{Q}f\otimes \psi_{Q}$. Since the nonholomorphic part of $\tilde{f}$ is supported on exponents of the form $-\delta_{j}Q^{2}n^{2}$, the harmonic Maass form $\hat{f} := -\frac{1}{2} \pfrac{-1}{Q} \,\tilde{f}\otimes \psi_{Q}$ has no nonholomorphic part. Therefore
\[
	\hat{f} \in M_{\frac{1}{2}}^{!} (\Gamma_{0}(4N Q^{3}, \chi)).
\]
The $n$th coefficient of $\hat{f}$ is 
\[
	-\frac{1}{2}\pfrac{-1}{Q}\pfrac{n}{Q} \left(1 - \pfrac{-1}{Q} \pfrac{n}{Q}\right)a^{+}(n) = \hat{a}(n).
\]
That is, $\hat a(n)$ is $a^+(n)$ if $\pfrac{-n}{Q}=-1$ and is 0 otherwise.

To prove ($ii$), apply Theorem \ref{thm:treneer} to the weakly holomorphic modular form $\hat{f}$.

Statement ($iii$) is immediate from Theorem \ref{thm:lichtenstein}.  \qedhere
\end{proof}

Theorem \ref{thm:mocktheta} now follows quickly from Theorem \ref{thm:main}.
\begin{proof}[Proof of Theorem \ref{thm:mocktheta}]
	Using Theorem \ref{thm:main} and taking $m=\alpha$ we obtain congruences of the form
\[
	\hat{a}(p^{3}\ell^{m}n) \equiv 0 \pmod{\ell^{j}}
\]
for all $n$ coprime to $\ell p$. We may choose $Q$ in Theorem \ref{thm:main} to also be coprime to $\ell$. Then since $\hat{a}(k)=b(n)$ when $\pfrac{-k}{Q} = -1$, we can take any integer $A$ satisfying $(A,p\ell)=1$ and
\[
	\pfrac{-p^{3}\ell^{m}A}{Q} = -1
\]
so that, replacing $n$ by $p \ell Q n + A$, we obtain the congruences
\[
	b(p^{4}\ell^{m+1}Q n + B) \equiv 0 \pmod{\ell^{j}},
\]
where $B=p^{3}\ell^{m} A$.
\end{proof}

\begin{rem*}
	For the case where $\alpha=0$, Theorem \ref{thm:mocktheta} can be improved. Namely, we can have $Q=\ell$ and, with $A$ chosen as above, we can similarly obtain congruences $b(p^{4}\ell n + B) \equiv 0 \pmod{\ell^{j}}$ where $B=p^3 A$.
\end{rem*}

\begin{rem*}
	Congruences for a Ramanujan mock theta function $M(q)=\sum c(n)q^n$ follow from Theorem \ref{thm:mocktheta} in the form $c(p^{4}\ell^{\alpha+1} Q n + B') \equiv 0 \pmod{\ell^{j}}$, where $B'=(B-\tau)/\delta\in \Q$. If $(\delta,p Q\ell)=1$, $A$ may be chosen as above to also satisfy $A\equiv \tau\ell^{-\alpha}p^{-3}\pmod{\delta}$ to get $B'\in\N$.
\end{rem*}

\bibliographystyle{plain}
\bibliography{bibliography}

\end{document}